\documentclass[reqno,10pt]{amsart}
\usepackage{amssymb}
\usepackage[english]{babel}
\usepackage[T2A]{fontenc}
\usepackage{color}

\usepackage{graphicx}
\usepackage{graphics}
\usepackage[colorlinks]{hyperref}

\usepackage{epstopdf}
\usepackage[margin=1.0in]{geometry}
\numberwithin{equation}{section}

\newtheorem{theorem}{Theorem}[section]
\newtheorem{lemma}[theorem]{Lemma}
\newtheorem{remark}[theorem]{Remark}
\newtheorem{corollary}[theorem]{Corollary}
\newtheorem{proposition}[theorem]{Proposition}
\newcommand{\diver}{\mathop{\rm div}}

\begin{document}

\title[Maximal existence domains]{Maximal existence domains\\  
of positive solutions\\
for two-parametric systems of elliptic equations}

\author[Vladimir Bobkov, Yavdat Il'yasov]
{Vladimir Bobkov, Yavdat Il'yasov} 

\address{Yavdat Il'yasov \newline
Institute of Mathematics of RAS,
Ufa, Russia}
\email{ilyasov02@gmail.com}

\address{Vladimir Bobkov \newline
Institute of Mathematics of RAS,
Ufa, Russia \newline
}
\email{bobkovve@gmail.com}

\thanks{The first author was 
partly supported by grant RFBR 
13-01-00294-p-a. The second author was 
partly supported by grant RFBR
14-01-00736-p-a}

\begin{abstract}
The paper is devoted to the study of two-parametric  families of Dirichlet problems for systems of equations with $p, q$-Laplacians and  indefinite nonlinearities. 
Continuous and monotone curves $\Gamma_f$ and $\Gamma_e$ on the parametric plane $\lambda\times \mu$, which are the lower and upper bounds for a maximal domain of existence of weak positive solutions are introduced.
The curve $\Gamma_f$ is obtained by developing our previous work \cite{BobkovIlyasov} and it determines a maximal domain of the applicability of the Nehari manifold and fibering methods.
The curve $\Gamma_e$ is derived explicitly via minimax variational principle of the extended functional method.
\end{abstract}

\maketitle

\section{Introduction}
We consider the Dirichlet problem for system of equations
\begin{equation}
\tag{$\mathcal{D}$}
\label{D1}
\left\{
\begin{aligned}
  -\Delta_p u  &= \lambda |u|^{p-2} u + c_1 f(x) |u|^{\alpha-2} |v|^{\beta}  u , &&x \in \Omega, \\[0.4em]
  -\Delta_q v &= \mu |v|^{q-2} v + c_2 f(x) |u|^{\alpha} |v|^{\beta-2} v, &&x \in \Omega, \\[0.4em]
	\left. u \right|_{\partial \Omega} &=  \left. v \right|_{\partial \Omega} = 0,
\end{aligned} \right.
\end{equation}
where  $\Omega$ is a bounded domain in  $\mathbb{R}^n$, $n \geq 1$, with the boundary $\partial \Omega$ which is
$C^{1,\delta}$-manifold, $\delta \in (0, 1)$; parameters $\lambda, \mu \in \mathbb{R}$, $c_1, c_2 > 0$; $p, q > 1$ and $\alpha,~ \beta \geq 1$; the function $f \in L^\infty(\Omega)$ and possibly changes the sign.

The $p$- and $q$-Laplacians in \eqref{D1} are the special cases of divergence-form operator $\diver (a(x, \nabla u))$, which appears in many nonlinear diffusion problems (cf. \cite{diaz, yin} for a discussion of some physical backgrounds for elliptic systems with $p$-Laplacians). 
The main feature of the problem \eqref{D1} is that the nonlinearity on the right-hand side has a priori indefinite sign. Problems with such nonlinearities possess complicated and interesting geometrical structure of branches of solutions, for example, the blow-up behavior of branches at critical values of parameters, existence of turning points, etc. (see, e.g. \cite{Alam, ilIZV, Ou2}).
 
In the present article, we study a \textit{maximal existence domain} of nonnegative (positive) solutions to \eqref{D1}. 
By the maximal existence domain we mean a set in the $\lambda \times \mu$-plane for which \eqref{D1} possesses weak nonnegative (positive) solutions, whereas there are no such solutions in its complement. 
The problem of finding and description domains with such extremal properties, as well as their precise boundaries, is fundamental in the theoretical investigation of  parametric problems and it naturally arises in various  applications. 
Some results in this direction for the system \eqref{D1} can be found, for example, in \cite{Bozhkov, yang2008, BobkovIlyasov}. 
In particular, the existence of a nonnegative solution in quadrant $\{\lambda < \lambda_1\} \times \{\mu<\mu_1\}$ and in a neighbourhood of $(\lambda_1, \mu_1)$ in quadrant $\{\lambda > \lambda_1\} \times \{\mu > \mu_1\}$ had been obtained in \cite{Bozhkov,yang2008}.
Here $\lambda_1$ and $\mu_1$ denote the first Dirichlet eigenvalues  of $-\Delta_p$ and $-\Delta_q$ in $\Omega$, respectively.
Furthermore, in \cite{BobkovIlyasov} it was introduced explicitly a maximal value $\sigma^*$ of applicability of the Nehari manifold and fibering methods on a ray $(\lambda_1 \sigma, \mu_1 \sigma)$, $\sigma \geq 1$, which defines a nonlocal domain $(\lambda_1, \sigma^* \lambda_1) \times (\mu_1, \sigma^* \mu_1)$ of existence of nonnegative solutions to \eqref{D1}.

We are concerned with the investigation of the maximal existence domain of nonnegative solutions for the problem \eqref{D1} and its boundary in quadrant $\{\lambda > 0\} \times \{\mu > 0\}$. The part of this boundary can be described as a parametrized set $\Gamma = \left(\lambda^*(r), \mu^*(r) \right)$, $r>0$, where an extremal point $(\lambda^*(r), \mu^*(r))$ can be defined implicitly as follows:
\begin{equation}
\label{lsol}
\lambda^*(r) := \sup \left\{ \lambda \geq 0: \eqref{D1} ~\text{ with }~ (\lambda, \mu) = (\lambda, \lambda \, r) ~\text{ has a nonnegative solution} \, \right\}, \quad \mu^*(r) := \lambda^*(r) \, r.
\end{equation}
Our goal is to construct a lower estimate $\Gamma_f = (\lambda^*_f(r), \mu^*_f(r))$ and an upper estimate $\Gamma_e = (\lambda^*_e(r), \mu^*_e(r))$ of $\Gamma$, such that $\lambda^*_f(r) \leq \lambda^*(r) \leq \lambda^*_e(r)$ and $\mu^*_f(r) \leq \mu^*(r) \leq \mu^*_e(r)$. To shorten notation, we write $\Gamma_f \leq \Gamma \leq \Gamma_e$ in this case.

To find the lower estimate we develop the results obtained in \cite{BobkovIlyasov}
and define the curve $\Gamma_f$ by explicit variational formulas. 
These formulas allow us to show that $\Gamma_f$ is a continuous monotone curve and it allocates a maximal domain in quadrant $\{\lambda >\lambda_1 \} \times \{\mu > \mu_1 \}$ where the fibering and Nehari manifold methods are applicable  to  \eqref{D1}. 
The upper estimate $\Gamma_e$ for the maximal existence domain of positive solutions is obtained by applying the extended functional method \cite{ilfunc} to \eqref{D1}. 
Using this approach we describe $\Gamma_e$ explicitly via minimax variational principle and prove that it is also a continuous monotone curve. 

It is noteworthy that the variational principles used to describe these critical curves make it possible to approximate them numerically (see, e.g., \cite{IvaIly, LefWei}). Furthermore, we suppose that the obtained explicit formulas may be useful in the investigation of the nonstationary problems corresponding to \eqref{D1} (cf. \cite{IlD}) and problems with supercritical nonlinearities (see, e.g., \cite{ilrunSup}). 

The paper is organized as follows. 
In Section 2, we present the main results of the paper.
Section 3 deals with the lower bound curve $\Gamma_f$.
In Section 4, we treat the upper bound curve $\Gamma_e$.
Appendices contain some technical statements and regularity results.

\section{Main results}

Let us introduce some notations. We denote
$\Omega^+ :=\{x\in \Omega: f(x)> 0\}$ and $\Omega^0 :=\{x\in \Omega: f(x)=0\}$.
By $\nu(U)$ we mean the Lebesgue measure of a set $U \subset \mathbb{R}^n$, and say that  $U$ has nonempty interior a.e. if it contains an open subset, after redefinition on a set of measure zero. 
The standard Sobolev spaces $W_0^{1,p} := W_0^{1,p}(\Omega)$ and $W_0^{1,q} := W_0^{1,q}(\Omega)$ are equipped with the norms
$$
||u||_p := \left(\int_\Omega |\nabla u|^p \, dx\right)^{1/p},  \quad ||v||_q := \left(\int_\Omega |\nabla v|^q \, dx\right)^{1/q}.
$$
By $p^*$ and $q^*$ we denote the critical exponents of  $W^{1,p}$ and $W^{1,q}$; $(\lambda_1, \varphi_1)$ and $(\mu_1, \psi_1)$ stand for the first eigenpairs of the operators $-\Delta_p$ and $-\Delta_q$ in $\Omega$ with the zero Dirichlet data, respectively. 
It is known that $\lambda_1$, $\mu_1$ are positive, simple and isolated, and $\varphi_1$,  $\psi_1 \in C^{1, \delta}(\overline{\Omega})$, $\delta \in (0,1)$ are positive \cite{Anan}.

We call  $(u,v) \in  W_0^{1,p} \times W_0^{1,q}$  a \textit{weak solution (sub-, supersolution)} of the problem \eqref{D1} if $f |u|^{\alpha-2} |v|^{\beta}  u \in W^{-1,p'}(\Omega)$, $f |u|^{\alpha} |v|^{\beta-2}  v \in W^{-1,q'}(\Omega)$ and for all nonnegative $(\xi, \eta) \in  W_0^{1,p} \times W_0^{1,q}$ 
\begin{equation}
\label{defin}
  \begin{aligned}
\int_\Omega |\nabla u|^{p-2} \nabla u \nabla \xi \,dx - \lambda \int_\Omega |u|^{p-2} u \xi \,dx - c_1 \int_\Omega f |u|^{\alpha-2} |v|^{\beta} u \xi \,dx &=  0, \quad &(\leq, \geq)\\
\int_\Omega |\nabla v|^{q-2} \nabla v \nabla \eta \,dx- 
\mu \int_\Omega |v|^{q-2} v \eta \,dx - c_2 \int_\Omega f |u|^{\alpha} |v|^{\beta-2} v \eta \,dx &=  0. \quad &(\leq, \geq)
 \end{aligned}
\end{equation}
Here $p'$ and $q'$ are conjugate exponents to $p$ and $q$, respectively.
Note that due to the standard approximation theory, it is sufficient to consider only nonnegative $\xi, \eta \in C^1_0(\Omega)$.
We say that a weak solution $(u, v)$ of \eqref{D1} is \textit{$C^1$-solution} if $u, v \in C^1(\overline{\Omega})$.
Due to regularity results (see Lemma \ref{lem:l} and Corollary \ref{rem:c}), any weak solution $(u, v)$ of \eqref{D1} is $C^1$-solution, provided $\alpha, \beta \geq 1$ and $\frac{\alpha}{p^*} + \frac{\beta}{q^*} < 1$.
We say that a weak solution $(u, v)$ is \textit{positive}, whenever $u,v > 0$ a.e. in  $\Omega$. If $\alpha \geq p$, $\beta \geq q$ and $\frac{\alpha}{p^*} + \frac{\beta}{q^*} < 1$, then Lemma \ref{lem:l} and Corollary \ref{rem:>0} imply that any nontrivial nonnegative solution is positive.

Notice also that the pairs $(\varphi_1, 0)$ and $(0, \psi_1)$ are semi-trivial solutions of \eqref{D1}, when $(\lambda, \mu)$ belongs to the lines $\lambda_1 \times \mathbb{R}$ and $\mathbb{R} \times \mu_1$, respectively.
In what follows, solution $(u, v)$ of \eqref{D1} with $u, v \not\equiv 0$ in $\Omega$ will  be called \textit{nontrivial}.

Let us introduce the maximal existence domain of nonnegative solutions for \eqref{D1} in quadrant $\{\lambda >\lambda_1 \} \times \{\mu > \mu_1 \}$:
\begin{align*}
\Sigma^* := \bigcup_{r > 0} \left\{ (\lambda, \mu) \in \mathbb{R}^2:  ~\lambda_1 < \lambda < \lambda^*(r), ~ \mu_1 < \mu < \mu^*(r) \right\}.
\end{align*}

\subsection{Lower bound for \textnormal{(\ref{lsol})} in quadrant \texorpdfstring{$\{\lambda >\lambda_1 \} \times \{\mu > \mu_1 \}$}{lambda>lambda1,mu>mu1}}
In this part, we use the assumptions
\begin{equation}
\label{Sob}
\alpha,~ \beta \geq 1, \qquad \frac{\alpha}{p} +\frac{\beta}{q} > 1, \qquad \frac{\alpha}{p^*} +\frac{\beta}{q^*} < 1,	
\end{equation}

To obtain the lower estimate for \eqref{lsol} in quadrant $\{\lambda >\lambda_1 \} \times \{\mu > \mu_1 \}$ we introduce the set of critical points
\begin{equation}\label{sf}
	\lambda^*_f(r) := \inf_{(u,v) \in W} \left\{ \max \left\{ \frac{\int_\Omega |\nabla u|^p \, dx}{\int_\Omega |u|^p \, dx},~  \frac{1}{r} \frac{\int_\Omega |\nabla v|^q \, dx}{\int_\Omega |v|^q \, dx} \right\}: F(u,v) \geq 0 \right\}, \quad \mu^*_f(r) := \lambda^*_f(r) \, r,
\end{equation}
parametrized by $r > 0$. Here $W:= W^{1,p}_0 \times W^{1,q}_0 \setminus \{ (0, 0) \}$ and
$$
F(u,v) := \int_\Omega f(x)|u|^{\alpha}|v|^{\beta}\, dx.
$$
The family \eqref{sf} generalizes the single critical value $\sigma^*$ obtained in \cite{BobkovIlyasov}, for which we have $\sigma^* = \frac{1}{\lambda_1}\lambda^*_f\left(\frac{\mu_1}{\lambda_1}\right)$.
The family \eqref{sf} forms a curve $\Gamma_f(r) = (\lambda^*_f(r), \mu^*_f(r))$, $r>0$, which  allocates the following set in quadrant $\{\lambda >\lambda_1 \} \times \{\mu > \mu_1 \}$:
\begin{align*}
\Sigma^*_f := \bigcup_{r > 0} \left\{ (\lambda, \mu) \in \mathbb{R}^2:  ~\lambda_1 < \lambda < \lambda^*_f(r), ~ \mu_1 < \mu < \mu^*_f(r) \right\}.
\end{align*}
The main results on $\Gamma_f$ and $\Sigma^*_f$ are contained  in the following theorem.
\begin{theorem}
\label{Th3}
Assume \eqref{Sob} is satisfied.
Then
\begin{enumerate}
\item $(\lambda_1, \mu_1) \leq \Gamma_f(r) < +\infty$ for all $r \in (0, +\infty)$.
\item $\Gamma_f(r)$ is continuous for all $r \in (0, +\infty)$.
\item There exists $r \in (0, +\infty)$ such that $\Gamma_f(r) > (\lambda_1, \mu_1)$ if and only if $F(\varphi_1, \psi_1) < 0$;
\item $\lambda^*_f(r)$ is non-increasing and $\mu^*_f(r)$ is non-decreasing on $(0, +\infty)$.
\item for any $(\lambda, \mu) \in \Sigma^*_f$ there exists a nonnegative $C^1$-solution of \eqref{D1}.
\end{enumerate}
\end{theorem}

Here and subsequently, $\Gamma_f(r) < +\infty$ means $\lambda^*_f(r), \mu^*_f(r) < +\infty$, and $\Gamma_f(r) > (\lambda_1, \mu_1)$ means $\lambda^*_f(r) > \lambda_1$ and $\mu^*_f(r) > \mu_1$.

We stress that statement (5) of Theorem \ref{Th3} implies $\Sigma^*_f \subseteq \Sigma^*$ and therefore $\Gamma_f$ is the lower bound for $\Gamma$ in quadrant $\{\lambda >\lambda_1 \} \times \{\mu > \mu_1 \}$, i.e. $\Gamma_f \leq \Gamma$. 
Moreover, statement (3) of Theorem \ref{Th3} implies that the assumption $F(\varphi_1, \psi_1) < 0$ is sufficient for nonemptyness of $\Sigma^*$. Similar to the scalar analog of \eqref{D1} (see \cite{ilrunSup}), we suspect that this assumption is also necessary.
In Section 4, it will be shown that it is meaningful to call $\Sigma^*_f$ a maximal domain of applicability of the Nehari manifold and fibering methods in quadrant $\{\lambda >\lambda_1 \} \times \{\mu > \mu_1 \}$.  
Nevertheless, it should be undertaken that $\Sigma^*_f$, in general, is not a maximal domain of existence of nonnegative solutions for \eqref{D1}, i.e. $\Sigma^*_f \neq \Sigma^*$ (see \cite[Section 10]{BobkovIlyasov}).

The definition \eqref{sf} implies that $\lambda^*_f(r)$ doesn't depend on parameters $c_1, c_2 > 0$ of \eqref{D1}, and therefore $\Gamma_f$ is invariant under a change of  $c_1, c_2$.
In Section 4, we study also a behaviour of $\Gamma_f$. Namely,  the variational principle \eqref{sf} allows us to provide the precise information on asymptotics of $\Gamma_f$ for $r \to 0$ and $r \to +\infty$, see Lemmas \ref{lemma:2}, \ref{lemma:3} and Figs.~1, 2.

\subsection{Upper bound for \textnormal{(\ref{lsol})}}
In this part we obtain an upper estimate for the boundary of a maximal existence domain of \textit{positive} $C^1$-solutions for \eqref{D1} by means of the extended functional method \cite{ilfunc}.
For this purpose we consider the family of the extended functionals 
$\Phi_{(\lambda,\lambda \,r)} :(S \times S)\times (\overline{S}\times \overline{S} ) \to \mathbb{R}$
which correspond to \eqref{D1} and are defined as
\begin{align}
\notag
 \Phi_{(\lambda, \lambda \,r)}(u,v; \xi,\eta) &:= 
\int_\Omega |\nabla u|^{p-2} \nabla u \nabla \xi \, dx - 
\lambda \int_\Omega |u|^{p-2} u \xi \, dx \\
\notag
&+ 
\int_\Omega |\nabla v|^{q-2} \nabla v \nabla \eta \, dx - 
\lambda \, r \int_\Omega |v|^{q-2} v \eta \, dx  \\
&- 
c_1 \int_\Omega f |u|^{\alpha-2} |v|^\beta u \xi \, dx -
c_2 \int_\Omega f |u|^{\alpha} |v|^{\beta-2} v \eta \, dx, \quad r > 0.
\end{align} 
Here
\begin{align*}
&S:= \{w \in C^1(\overline{\Omega}):~ w > 0 \mbox{ in } \Omega, ~w = 0 \mbox{ on } \partial \Omega \}, \\
&\overline{S}:= \{w \in C^1(\overline{\Omega}):~ w \geq 0 \mbox{ in } \Omega, ~w = 0 \mbox{ on } \partial \Omega \}.
\end{align*}
Resolving the equation $\Phi_{(\lambda, \lambda \, r)}(u, v; \xi, \eta) = 0$ with respect to $\lambda$, we obtain
\begin{align}
\notag
\mathcal{L}_r&(u, v; \xi, \eta) := \lambda \\
\label{Efm2}
&= \frac{\int_\Omega |\nabla u|^{p-2} \nabla u \nabla \xi \, dx +  \int_\Omega |\nabla v|^{q-2} \nabla v \nabla \eta \, dx
- 
c_1 \int_\Omega f |u|^{\alpha-2} |v|^\beta u \xi \, dx -
c_2 \int_\Omega f |u|^{\alpha} |v|^{\beta-2} v \eta \, dx }
{\int_\Omega |u|^{p-2} u \xi \, dx + r \int_\Omega |v|^{q-2} v \eta \, dx}.
\end{align}
Henceforth we will assume also that $\{ (0, 0) \} \not\in \overline{S}\times \overline{S}$ to circumvent the case of zero denominator in \eqref{Efm2}.

Now following \cite{ilfunc} we introduce for each $r>0$ the extended functional critical points
\begin{align}
\label{extcrit}
\lambda^*_e(r) := \sup_{u, v \in S} \inf_{\xi, \eta \in \overline{S}} \mathcal{L}_r(u, v; \xi, \eta), \quad \mu^*_e(r) := \lambda^*_e(r) \, r,
\end{align}
which form a curve $\Gamma_e(r) = \left( \lambda^*_e(r), \mu^*_e(r) \right)$, $r>0$.
The main properties of $\Gamma_e$ are given in the following theorem.  
\begin{theorem} \label{Th0}
Assume $\alpha,\beta \geq 1$.  Then 
\begin{enumerate}
\item $\Gamma_e(r)  > -\infty$ for all $r \in (0, +\infty)$.
\item If $\Omega^0 \cup \Omega^+$ has nonempty interior a.e., then $\Gamma_e(r)<+\infty$ for all $r \in [0,+\infty]$.
\item If $p = q$ and  $\nu(\Omega^0 \cup \Omega^+) = 0$, then $\Gamma_e(r) = +\infty$ for all $r \in (0,+\infty)$.
\item If $\Gamma_e(r_0) > (0, 0)$  for some $r_0>0$, then $\Gamma_e(r)>(0, 0)$ for all $r \in (0, +\infty)$.
\item If $\Gamma_e(r_0) <+\infty$  for some $r_0>0$, then $\Gamma_e(r) <+\infty $ for all $r \in (0, +\infty)$.
\item If $(0, 0) < \Gamma_e(r) < +\infty$ for all $r \in (0,+\infty)$, then $\Gamma_e$ is continuous on $(0, +\infty)$.
\item If $(0, 0) < \Gamma_e(r) < +\infty$ for all $r \in (0,+\infty)$, then $\lambda^*_e(r)$ is non-increasing and $\mu^*_e(r)$ is non-decreasing on $(0, +\infty)$.
\end{enumerate}
\end{theorem}
These results are sketchily depicted on Figs.~1, 2. 
We suppose that statement (3) holds for any $p, q > 1$.
Analogously to $\Gamma_f$ we show an additional property of the curve $\Gamma_e$, namely the invariance of $\Gamma_e$ under a change of parameters $c_1, c_2$.
\begin{lemma}\label{invar}
Let $\alpha, \beta \geq 1$ and $\frac{\alpha}{p} + \frac{\beta}{q} \neq 1$. Then $\Gamma_e$ is independent of $c_1, c_2 > 0$.
\end{lemma}

Let us introduce the following sets:
\begin{align*}
\mathcal{R} &:= \bigcup_{r>0} \left\{ (\lambda, \mu) \in \mathbb{R}^2:~ \lambda > \lambda^*_e(r),~ \mu > \mu^*_e(r) \right\}, \\
\Sigma^*_e &:= \bigcup_{r>0} \left\{ (\lambda, \mu) \in \mathbb{R}^2:~ \lambda_1 < \lambda < \lambda^*_e(r),~ \mu_1 < \mu < \mu^*_e(r) \right\}.
\end{align*}
Let us note that $\mathcal{R} \cap \Sigma^*_e = \emptyset$. If $\mathcal{R}$ or $\Sigma^*_e$ is empty, then the assertion is obvious. 
Assume that $\mathcal{R}, \Sigma^*_e \neq \emptyset$ and suppose a contradiction, i.e. there exist $r_1, r_2 > 0$, such that
$$
\left\{ (\lambda, \mu) \in \mathbb{R}^2:~ \lambda > \lambda^*_e(r_1),~ \mu > \mu^*_e(r_1) \right\} \cap \left\{ (\lambda, \mu) \in \mathbb{R}^2:~ \lambda_1 < \lambda < \lambda^*_e(r_2),~ \mu_1 < \mu < \mu^*_e(r_2) \right\} \neq \emptyset.
$$
Thus, we can find $(\lambda, \mu) \in \mathbb{R}^2$, such that $
\lambda^*_e(r_1) < \lambda < \lambda^*_e(r_2)$ and 
$\mu^*_e(r_1) < \mu < \mu^*_e(r_2)$. However, these inequalities can not be satisfied simultaneously, due to statement (7) of Theorem \ref{Th0}.

Note that in view of statements $(5)$, $(6)$ of Theorem \ref{Th0} one has  $\mathcal{R}\neq \emptyset$ if and only if $\Gamma_e(r_0) < +\infty$ for some $r_0 \in (0,+\infty)$, and  $\Sigma^*_e \neq \emptyset$ if and only if $\Gamma_e(r_0) > (\lambda_1, \mu_1)$ for some $r_0 \in (0,+\infty)$. 
From these observations and the fact  $\mathcal{R} \cap \Sigma^*_e = \emptyset$ it is not hard to conclude that 
$\Gamma_e$ separates the sets $\mathcal{R}$ and $\Sigma^*_e$ in quadrant $\{\lambda > \lambda_1\}\times \{\mu > \mu_1\}$.

\begin{figure}[!ht]
\begin{minipage}[t]{0.47\linewidth}
\center{\includegraphics[scale=0.8]{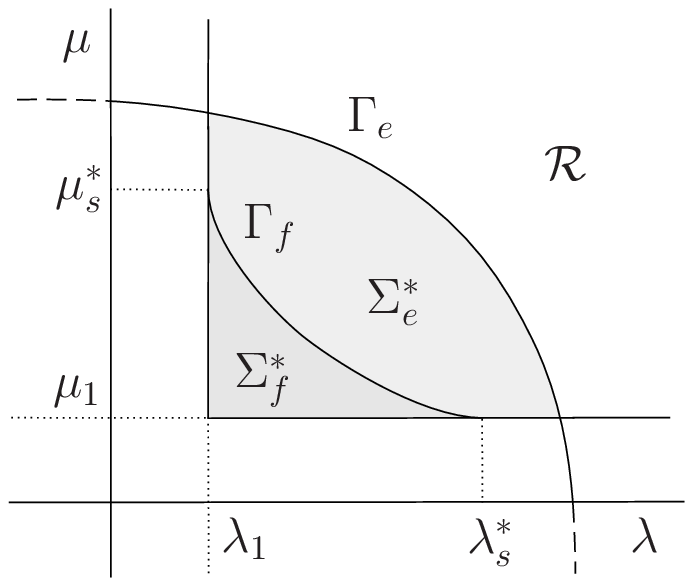}}\\
\caption{$\Omega^+ \cup \Omega^0$ has nonempty interior and $F(\varphi_1, \psi_1) < 0$}
\end{minipage}
\hfill
\begin{minipage}[t]{0.47\linewidth}
\center{\includegraphics[scale=0.8]{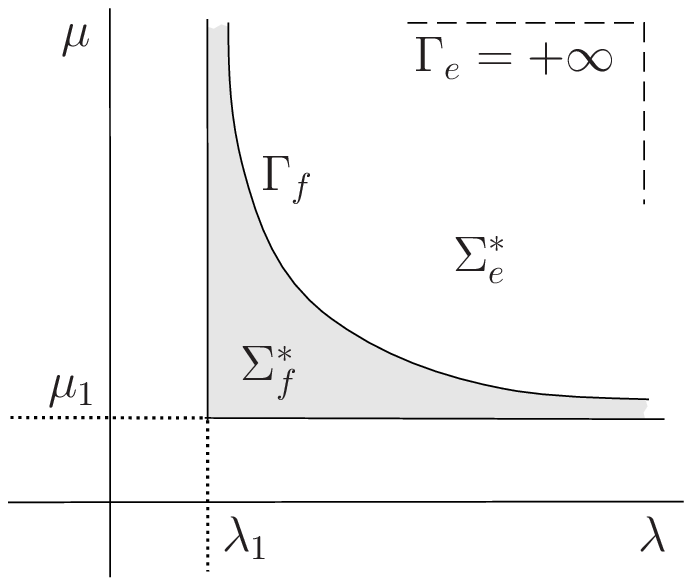}} \\
\caption{$\nu(\Omega^+ \cup \Omega^0) = 0$}
\end{minipage}
\end{figure}

The main results on $\mathcal{R}$ and  $\Sigma^*_e$  are given in  the following theorem.
\begin{theorem}
\label{Th1}
Assume $\alpha,\beta \geq 1$. Then 
\begin{enumerate}
	\item  \eqref{D1} has no positive  $C^1$-solutions for any $(\lambda, \mu) \in \mathcal{R}$; 
	\item  \eqref{D1} has a positive $C^1$-supersolution  for any $(\lambda, \mu) \in \Sigma^*_e$. 
\end{enumerate}
\end{theorem}

Statement (1) of Theorem \ref{Th1} implies that $\Gamma_e$ is the upper bound for a maximal existence domain of \textit{positive} $C^1$-solutions for \eqref{D1}.
Moreover, under the assumptions $\alpha \geq p$, $\beta \geq q$ and $\frac{\alpha}{p^*} + \frac{\beta}{q^*} < 1$, $\Gamma_e$ is the upper bound for $\Gamma$, i.e.
$\Gamma \leq \Gamma_e$. Indeed, in this case any nonnegative weak solution of \eqref{D1} is positive and of class $[C^1(\overline{\Omega})]^2$ (see Appendix B).
This fact also yields $\Gamma_f \leq \Gamma \leq \Gamma_e$ in quadrant $\{\lambda >\lambda_1 \} \times \{\mu > \mu_1 \}$, i.e. $\Sigma^*_f \subseteq \Sigma^* \subseteq \Sigma^*_e$.

The existence of supersolutions for \eqref{D1} in $\Sigma^*_e$ can be complemented by the existence of subsolutions for  $\lambda > \lambda_1$, $\mu > \mu_1$, which can be easily constructed using $\varphi_1$ and $\psi_1$. However,  in contrast to scalar equations, the existence of sub- and supersolutions in the sense of definition \eqref{defin} is not enough, in general, to obtain a solution for systems of elliptic equations (cf. \cite[p.~999]{Sattinger}).
Nevertheless, we suppose that $\Gamma_e$ is the \textit{precise} boundary of the maximal existence domain of positive $C^1$-solutions for \eqref{D1} and it coincides with $\Gamma$ in the case $\alpha \geq p$, $\beta \geq q$. 
The partial confirmation of this conjecture follows from Lemma \ref{lem:solution}. Moreover, for the scalar problems related to \eqref{D1} the extended functional critical point determines the sharp boundary of maximal existence interval of positive solutions \cite{ilrunSup}.

Summarizing the results of Theorems \ref{Th3} and \ref{Th0}, we get the following estimates for $\lambda^*(r)$ and $\mu^*(r)$, which are vector counterparts of \cite[Theorem 1.1]{ilrunSup} (see Fig. 1, 2).
\begin{corollary}
\label{Th2}
Assume that $\alpha \geq p$, $\beta \geq q$, $\frac{\alpha}{p^*}+\frac{\beta}{q^*} < 1$,     $\Omega^0 \cup \Omega^+$ and $F(\varphi_1, \psi_1) < 0$ has nonempty interior a.e. 
Then there exist $r \in (0, +\infty)$ such that
\begin{align*}
\lambda_1 < \lambda^*_f(r) \leq \lambda^*(r) \leq \lambda^*_e(r) < +\infty, \\
\mu_1 < \mu^*_f(r) \leq \mu^*(r) \leq \mu^*_e(r) < +\infty,
\end{align*}
or, equivalently, $(\lambda_1, \mu_1) < \Gamma_f(r) \leq \Gamma(r) \leq \Gamma_e(r) < +\infty$ in quadrant $\{\lambda >\lambda_1 \} \times \{\mu > \mu_1 \}$.
\end{corollary}
\begin{corollary}
\label{Th2x}
Assume that $\alpha \geq p$, $\beta \geq q$, $\frac{\alpha}{p^*}+\frac{\beta}{q^*} < 1$  and $\nu(\Omega^0 \cup \Omega^+) = 0$. Then for all $r \in (0, +\infty)$ it holds
\begin{align*}
\lambda_1 < \lambda^*_f(r) \leq \lambda^*(r) \leq \lambda^*_e(r) \leq +\infty, \\
\mu_1 < \mu^*_f(r) \leq \mu^*(r) \leq \mu^*_e(r) \leq +\infty,
\end{align*}
or, equivalently, $(\lambda_1, \mu_1) < \Gamma_f(r) \leq \Gamma(r) \leq \Gamma_e(r) \leq +\infty$.
\end{corollary}

\medskip
Finally, we provide an information on  nonexistence of solutions to \eqref{D1}.
\begin{lemma}
\label{Th27}
Assume that \eqref{Sob} is satisfied.
\begin{enumerate}
\item If $f \leq 0$ a.e. in $\Omega$, then \eqref{D1} has no nontrivial weak  
solutions in $\mathbb{R}^2\setminus\{ \lambda \geq \lambda_1\}\times\{\mu \geq \mu_1 \}$.
\item If $\Omega^0 \cup \Omega^+$ has nonempty interior a.e., then there exist $\widetilde{\lambda} \geq \lambda_1$ and $\widetilde{\mu} \geq \mu_1$, such that \eqref{D1} has no positive weak solutions in $\{\lambda \geq \widetilde{\lambda}\}\times\{\mu \in \mathbb{R}\}$ and $\{\lambda \in \mathbb{R}\} \times \{\mu \geq \widetilde{\mu}\}$.
\item If $f \geq 0$ a.e. in $\Omega$, then  \eqref{D1} has no positive weak solutions in $\mathbb{R}^2\setminus\{ \lambda<\lambda_1\}\times\{\mu<\mu_1 \}$.
\end{enumerate}
\end{lemma}

\section{On the threshold of the fibering method}
In this section we study the lower bound curve $\Gamma_f$ defined by means of \eqref{sf}.
We prove Theorem \ref{Th3} and describe precisely an asymptotic behaviour of $\Gamma_f$.

First we mention the following result on the existence of a minimizer to \eqref{sf}, which will be used further:
\begin{proposition}\label{exCr}
Assume \eqref{Sob} is satisfied. Then there exists a nonzero minimizer $(u^*_r, v^*_r) \in W$ of \eqref{sf} for any $r > 0$. Moreover, $u^*_r, v^*_r \geq 0$ in $\Omega$.
\end{proposition}
\begin{proof}
The proof can be obtained in the same way as the proof of \cite[Proposition 4.2, p.~8]{BobkovIlyasov}.
\end{proof}

Let us prove Theorem \ref{Th3}.
\begin{proof}
(1) The lower bound follows from the estimate
\begin{align}
\notag
\lambda^*_f(r) &\geq  \inf_{(u,v) \in W} \left[ \max \left\{ \frac{\int_\Omega |\nabla u|^p \,dx}{\int_\Omega |u|^p \,dx}, \frac{1}{r} \frac{\int_\Omega |\nabla v|^q \,dx}{\int_\Omega |v|^q \,dx} \right\}\right]  \\
\label{eq:l2>l1}
&= \max \left\{ \frac{\int_\Omega |\nabla \varphi_1|^p \,dx}{\int_\Omega |\varphi_1|^p \,dx}, \frac{1}{r} \frac{\int_\Omega |\nabla \psi_1|^q \,dx}{\int_\Omega |\psi_1|^q \,dx} \right\} = 
\max \{ \lambda_1, \mu_1/r \}, \quad  r >0.
\end{align}
The upper estimate $\lambda^*_f(r) < +\infty$ is obvious, since the admissible set of \eqref{sf} is nonempty for all $r > 0$. Indeed, if $u \in W_0^{1,p}$ and $v \in W_0^{1,q}$ are nontrivial functions with disjoint supports, then $F(u, v) = 0$, and hence $(u, v)$ is an admissible point for \eqref{sf}.

(2)
It is sufficient to prove the continuity of $\lambda^*_f(r)$ on $(0, +\infty)$. Observe that  the minimization problem \eqref{sf} has identical  admissible set $\{F(u,v) \geq 0,  (u,v) \in W\}$ for all $r>0$. Furthermore, for any $(u,v) \in W$ we have: if $r\leq r_0$, then
\begin{align}
\label{monot1}
\max \left\{ \frac{\int_\Omega |\nabla u|^p \,dx}{\int_\Omega |u|^p \,dx}, \frac{1}{r}  \frac{\int_\Omega |\nabla v|^q \,dx}{\int_\Omega |v|^q \,dx} \right\}
&\geq 
\max \left\{ \frac{\int_\Omega |\nabla u|^p \, dx}{\int_\Omega |u|^p \, dx}, \frac{1}{r_0}  \frac{\int_\Omega |\nabla v|^q \, dx}{\int_\Omega |v|^q \, dx} \right\} 
\end{align}
and  
\begin{align}
\label{monot2}
\max \left\{ r \frac{\int_\Omega |\nabla u|^p \,dx}{\int_\Omega |u|^p \,dx},  \frac{\int_\Omega |\nabla v|^q \,dx}{\int_\Omega |v|^q \,dx} \right\}
&\leq 
\max \left\{r_0\frac{\int_\Omega |\nabla u|^p \,dx}{\int_\Omega |u|^p \,dx}, \frac{\int_\Omega |\nabla v|^q \,dx}{\int_\Omega |v|^q \,dx} \right\}. 
\end{align}
Hence, $\lambda^*_f(r) \geq \lambda^*_f(r_0)$ and $\mu^*_f(r) \leq \mu^*_f(r_0)$ if $r\leq r_0$. 
Therefore, at every $r_0 \in (0, +\infty)$ there exist one-sided limits of $\lambda^*_f(r)$ and $\mu^*_f(r)$, and
\begin{align*}
\lim_{r \uparrow r_0} \lambda^*_f(r) \geq &\lambda^*_f(r_0) \geq \lim_{r \downarrow r_0} \lambda^*_f(r), \\
\lim_{r \uparrow r_0} \lambda^*_f(r) \, r_0  = \lim_{r \uparrow r_0} \lambda^*_f(r) \, r \leq
&\lambda^*_f(r_0) \, r_0 \leq
 \lim_{r \downarrow r_0} \lambda^*_f(r) \,r = \lim_{r \downarrow r_0} \lambda^*_f(r) \,r_0.
\end{align*}
Comparing these chains of inequalities we see that one-sided limits are equal to the  value of $\lambda^*_f(r)$ for each $r \in (0, +\infty)$, and this fact establishes the desired continuity of $\Gamma_f$.

(3)
Assume first that $\Sigma^*_f \neq \emptyset$. This implies that there exists $r_0 > 0$ such that  $\lambda^*_f(r_0) > \lambda_1$ and $\mu^*_f(r_0) > \mu_1$. Suppose, contrary to our claim, that $F(\varphi_1, \psi_1) \geq 0$. Then $(\varphi_1, \psi_1)$ is an admissible point for the minimiration problem \eqref{sf}. However 
$$
\max \left\{ \frac{\int_\Omega |\nabla \varphi_1|^p \, dx}{\int_\Omega |\varphi_1|^p \, dx}, \frac{1}{r_0}  \frac{\int_\Omega |\nabla \psi_1|^q \, dx}{\int_\Omega |\psi|^q \, dx} \right\}=\max\{ \lambda_1, \mu_1/r_0 \}
$$
and therefore $\lambda^*_f(r_0) \leq \max\{ \lambda_1, \mu_1/r_0 \}$ and $\mu^*_f(r_0) \leq \max\{ \lambda_1 \, r_0, \mu_1 \}$. These facts lead to a contradiction.

Assume now $F(\varphi_1, \psi_1) < 0$. On the contrary, suppose that $\Sigma^*_f = \emptyset$. Then  $\lambda^*_f(r) = \max\{ \lambda_1, \mu_1/r\}$ for all $r>0$ and, 
in particular, $\lambda^*_f(\mu_1/\lambda_1) = \max\{ \lambda_1, \lambda_1\}= \lambda_1$. 
Hence, Proposition \ref{exCr} implies the existence of $(u,v) \in \{F(u,v) \geq 0,  (u,v) \in W\}$ such that
$$
\frac{\int_\Omega |\nabla u|^p \, dx}{\int_\Omega |u|^p \, dx} = \lambda_1 \quad \mbox{and} \quad  \frac{\int_\Omega |\nabla v|^q \, dx}{\int_\Omega |v|^q \, dx} = \mu_1.
$$
These equalities are true if and only if $u = \varphi_1$ and $v = \psi_1$ up to multipliers. However this yields a contradiction, since $F(\varphi_1, \psi_1) < 0$ by the assumption. 

(4) 
Monotonicity directly follows from \eqref{monot1} and \eqref{monot2}.

(5)
In order to prove the existence of a weak solution for \eqref{D1} in $\Sigma^*_f$ let us note that problem \eqref{D1} has the variational form with the corresponding energy functional
\begin{align*}
\notag
E_{\lambda, \mu}(u,v) &= \frac{\alpha}{c_1 p} \left( \int_\Omega |\nabla u|^p \, dx - \lambda \int_\Omega |u|^p \, dx \right) \\
&+\frac{\beta}{c_2 q} \left( \int_\Omega |\nabla v|^q \, dx- \mu \int_\Omega |v|^q \, dx \right) - \int_\Omega f(x)|u|^{\alpha}|v|^{\beta}\, dx.
\end{align*}
Moreover, Proposition \ref{semivar} implies that without loss of generality and for simplicity of notations we can deal with the case $c_1 = \alpha$, $c_2 = \beta$.

Following \cite{BobkovIlyasov}, we look for a weak nonnegative solution to \eqref{D1} as minimizer of the problem
\begin{equation}
\label{eq:c_lambda}
n_{\lambda, \mu} := \inf \{ E_{\lambda, \mu}(u,v): (u,v) \in \mathcal{N}_{\lambda, \mu} \},
\end{equation}
where 
\begin{align*}
\mathcal{N}_{\lambda, \mu} := \{(u,v) \in W:
~&P_{\lambda, \mu}(u, v) := \left< D_u E_{\lambda, \mu}(u,v), u \right> = 0, \\
 &Q_{\lambda, \mu}(u, v) := \left< D_v E_{\lambda, \mu}(u,v), v \right> = 0 \}
\end{align*}
is the Nehari manifold. Consider the Hessian of $E_{\lambda, \mu}(u,v)$:
$$
\mathcal{H}_{\lambda, \mu}(u,v) =  \left( \begin{array}{cc}
\left< D_{u}P_{\lambda, \mu}(u,v), u \right> & \left< D_{v}P_{\lambda, \mu}(u,v), v \right>  \\[0.8em]
\left< D_{u}Q_{\lambda, \mu}(u,v), u \right> & \left< D_{v}Q_{\lambda, \mu}(u,v), v \right>
\end{array} \right).
$$

The following two lemmas are the basis of the spectral analysis by the fibering method  (see \cite{BobkovIlyasov, ilIZV}).

\begin{lemma}\label{LS1}
Let $(\lambda, \mu) \in \mathbb{R}^2$ and  $(u_0, v_0) \in \mathcal{N}_{\lambda, \mu}$ be a minimization point of \eqref{eq:c_lambda} such that 
$$
\mbox{det}\, \mathcal{H}_{\lambda, \mu} (u_0, v_0) \neq 0.	
$$
Then $(u_0, v_0)$ is a critical point of $E_{\lambda, \mu}(u, v)$, i.e. a weak solution of \eqref{D1}.
\end{lemma} 
\begin{proof}
The proof is obtained by Lagrange multiplier rule, cf. \cite[Lemma 3.1 p.~6]{BobkovIlyasov}.
\end{proof}
\begin{lemma}\label{DCor}
Assume \eqref{Sob} is satisfied, $p,q \in (1, +\infty)$ and $f \in L^{\infty}(\Omega)$. If  $(\lambda, \mu) \in \Sigma^*_f$, then 
$\mbox{det}\, \mathcal{H}_{\lambda, \mu} (u, v) \neq 0$
for any $(u,v) \in \mathcal{N}_{\lambda, \mu}$.
\end{lemma}
\begin{proof}
The proof can be obtained by direct generalization of \cite[Corollary 4.5, p.~9]{BobkovIlyasov}
\end{proof}

Thus, these lemmas ensure to find a weak solution of \eqref{D1} by means of the minimization problem \eqref{eq:c_lambda}, whenever $(\lambda,\mu) \in \Sigma^*_f$. 
Moreover, it can be established analogically to the proof of \cite[Proposition 4.2, p.~8]{BobkovIlyasov} that \eqref{eq:c_lambda} indeed possesses a minimizer  $(u, v) \in \mathcal{N}_{\lambda, \mu}$, which is therefore a weak solution of \eqref{D1}.
The desired regularity of $u$, $v$ follows from Lemma \ref{lem:l} and Corollary \ref{rem:c}.
\end{proof}

\begin{remark}
Direct usage of \eqref{eq:c_lambda} for obtaining weak solutions to \eqref{D1} in the case $(\lambda, \mu) \not\in \Sigma^*_f$ is not possible, in general, since one can face with $\mbox{det}\, \mathcal{H}_{\lambda, \mu} (u, v) = 0$ for minimizer $(u,v)$ of \eqref{eq:c_lambda}. 
Therefore we call $\Sigma^*_f$ the maximal domain of applicability of the Nehari manifold and fibering methods in quadrant $\{\lambda >\lambda_1 \} \times \{\mu > \mu_1 \}$. 
%
\end{remark}


Let us now study the asymptotic behavior of $\Gamma_f$.
Introduce the following critical values
\begin{equation}
\label{lsms}
\lambda^*_s := \inf_{u \in W_0^{1,p} \setminus \{0\}} \left\{ \frac{\int_\Omega |\nabla u|^p \, dx}{\int_\Omega |u|^p \, dx}: F(u, \psi_1) \geq 0 \right\}, \quad 
\mu^*_s := \inf_{v \in W_0^{1,q} \setminus \{0\}} \left\{ \frac{\int_\Omega |\nabla v|^q \, dx}{\int_\Omega |v|^q \, dx}: F(\varphi_1, v) \geq 0 \right\},
\end{equation}
and define $r_0 := \mu_1/\lambda^*_s$, $r_1 := \mu^*_s/\lambda_1$.
\begin{lemma}
\label{lemma:2}
Let $F(\varphi_1, \psi_1) < 0$ and $\Omega^+ \cup \Omega^0$ has nonempty interior a.e. Then (see Fig.~1)
\begin{enumerate}
\item $\lambda^*_f(r) = \mu_1 / r$, $\mu^*_f(r) = \mu_1$ for all $r \in (0, r_0]$;
\item $\lambda^*_f(r) = \lambda_1$, $\mu^*_f(r) = \lambda_1 r$ for all $r \in [r_1, +\infty)$;
\item $\lambda^*_f(r) > \lambda_1$, $\mu^*_f(r) > \mu_1$ for all $r \in (r_0, r_1)$.
\end{enumerate}
\end{lemma}
\begin{proof}
(1) It is not hard to show that under the assumptions of the lemma $\lambda^*_s$ defined by \eqref{lsms} is finite (cf. the proof of \cite[Lemma 3.1, p.~34]{ilIZV}), 
with a corresponding minimizer $u^* \in W_0^{1,p}$ and $F(u^*, \psi_1) \geq 0$.
Taking $(u^*, \psi_1)$ as an admissible point for $\lambda^*_f(r)$ and noting \eqref{eq:l2>l1}, we get
$$
\frac{\mu_1}{r} \leq \lambda^*_f(r) \leq \max \left\{ \frac{\int_\Omega |\nabla u^*|^p \, dx}{\int_\Omega |u^*|^p \, dx}, \frac{1}{r} \frac{\int_\Omega |\nabla \psi_1|^q \, dx}{\int_\Omega |\psi_1|^q \, dx} \right\} = \max \left\{ \lambda^*_s, \frac{\mu_1}{r} \right\} \leq \frac{\mu_1}{r}
$$
for any $r \in (0, \mu_1/\lambda^*_s]$.  Hence, for such $r$ we have $\lambda^*_f(r) = \mu_1/r$ and consequently $\mu^*_f(r) = \mu_1$.

Statement (2) can be handled in much the same way. 

(3) Note that from statement (1) of Theorem \ref{Th3} we have $\lambda^*_f(r)\geq \lambda_1$ and $\mu^*_f(r)\geq \mu_1$. Suppose, contrary to our claim, that $\mu^*_f(r) = \mu_1$ for some $r > r_0=\mu_1/\lambda^*_s$. 
Then, in view of Proposition \ref{exCr}, there exists $u^*_r \in W_0^{1,p}$ such that $(u^*_r, \psi_1)$ is a minimizer of $\mu^*_f(r)$ and $F(u^*_r, \psi_1) \geq 0$.
Therefore, $u^*_r$ is an admissible point for $\lambda^*_s$, and
$$
\lambda^*_s \leq \frac{\int_\Omega |\nabla u^*_r|^p \,dx}{\int_\Omega |u^*_r|^p \,dx} \leq \lambda^*_f(r) =
\max \left\{ \frac{\int_\Omega |\nabla u^*_r|^p \,dx}{\int_\Omega |u^*_r|^p \,dx}, \frac{\mu_1}{r}  \right\} =  \frac{\mu_1}{r}.
$$
Hence, $r \leq \mu_1/\lambda^*_s$, but it contradicts our assumption.

By the same arguments it can be shown that $\lambda^*_f(r) > \lambda_1$ for any $r < r_1$. Thus, we conclude that  $\lambda^*_f(r) > \lambda_1$, $\mu^*_f(r) > \mu_1$ for all $r \in (r_0, r_1)$.
\end{proof}

\begin{lemma}
\label{lemma:3}
Let $\nu(\Omega^+ \cup \Omega^0) = 0$. Then (see Fig.~2)
\begin{enumerate}
\item $\lambda^*_f(r) > \lambda_1$, $\mu^*_f(r) > \mu_1$ for all $r > 0$;
\item $\lambda^*_f(r) \to \lambda_1$, $\mu^*_f(r) \to +\infty$ as $r \to +\infty$;
\item $\lambda^*_f(r) \to +\infty$, $\mu^*_f(r) \to \mu_1$ as $r \to 0$.
\end{enumerate}
\end{lemma}
\begin{proof}
(1) Suppose the assertion is false and, without loss of generality, $\mu^*_f(r) = \mu_1$ for some $r > 0$. Then using Proposition \ref{exCr} we can find $u^*_r \in W_0^{1,p}$ such that  $(u^*_r, \psi_1)$ is a nonzero minimizer of $\mu^*_f(r)$  and $F(u^*_r, \psi_1) \geq 0$. However, assumption $\nu(\Omega^+ \cup \Omega^0) = 0$ implies $F(u, \psi_1) < 0$ for any nontrivial $u \in W_0^{1,p}$. Thus we get a contradiction.
 In the same way it can be shown that $\lambda^*_f(r) > \lambda_1$ for all $r > 0$.

(2) The basic idea is to construct a suitable admissible point for minimization problem \eqref{sf}.  By definition of $W_0^{1,p}(\Omega)$ there exists $u_k \in C^\infty_0(\Omega)$, $k \in \mathbb{N}$, such that $u_k \to \varphi_1$ in $W^{1,p}(\Omega)$, that is,
$$
\int_\Omega |\nabla u_k|^p \,dx \to 
\int_\Omega |\nabla \varphi_1|^p \,dx \quad \mbox{and} \quad
\int_\Omega |u_k|^p \,dx \to 
\int_\Omega |\varphi_1|^p \,dx \quad \mbox{as } k \to +\infty.
$$
From here it follows, that
\begin{equation}
\label{eq:lk}
\lambda(u_k) := \frac{\int_\Omega |\nabla u_k|^p \,dx}{\int_\Omega |u_k|^p \,dx} \to \lambda_1, \quad k \to +\infty.
\end{equation}
Since $u_k$ has a  compact support in $\Omega$, there exists an open set $B_k$, such that $B_k \subset \Omega \setminus \mbox{supp}\,u_k$ for each $k \in \mathbb{N}$, and we denote by $\left(\mu_1(B_k), \psi_1(B_k) \right)$ the first eigenpair of $-\Delta_q$  in $B_k$ with zero Dirichlet boundary conditions. 
By construction, $\mbox{supp}\, u_k \cap \mbox{supp}\, \psi_1(B_k) = \emptyset$ for any $k \in \mathbb{N}$, and hence $F(u_k, \psi_1(B_k)) = 0$. Therefore, $(u_k, \psi_1(B_k))$ is an admissible point for $\lambda^*_f(r)$ and for all $r > \mu_1(B_k) / \lambda(u_k)$ we have
$$
\lambda_1 < \lambda^*_f(r) \leq \max \left\{ \lambda(u_k), \frac{\mu_1(B_k)}{r} \right\} \leq \lambda(u_k),
$$
where the first inequality is obtained from (1). Combining this fact with \eqref{eq:lk}, we deduce that $\lambda^*_f(r) \to \lambda_1$ and $\mu^*_f(r) \to +\infty$ as $r \to +\infty$.

The same method can be applied to prove statement (3).
\end{proof}

\section{Extended functional critical points}
In this section, we study the upper bound $\Gamma_e$ and prove Theorems \ref{Th0}, \ref{Th1}, \ref{Th27} and Lemma \ref{invar}.

\noindent
\textit{Proof of Theorem \ref{Th0}}.
(1) 
Let $\varphi \in W_0^{1,p}$ and $\psi \in W_0^{1,q}$ be weak solutions of
$$
\left\{
\begin{aligned}
-\Delta_p \varphi &= 1,  &&x \in \Omega, \\
\left. \varphi \right|_{\partial \Omega} &= 0,
\end{aligned}
\right.
\quad \mbox{and} \quad
\left\{
\begin{aligned}
-\Delta_q \psi &= 1,  &&x \in \Omega, \\
\left. \psi \right|_{\partial \Omega} &= 0,
\end{aligned}
\right.
$$
respectively.
Note that such solutions exist due to the coercivity of the corresponding energy functionals and their weakly lower semicontinuity on $W_0^{1,p}$ and $W_0^{1,q}$, respectively. Using the standard bootstrap arguments (cf. \cite[Lemma 3.2, p.~114]{drabekbook}) we get $\varphi, \psi \in L^\infty(\Omega)$. Hence $\varphi, \psi \in C^{1, \delta}(\overline{\Omega})$ with some $\delta \in (0,1)$ by \cite{Lieberman} and $\varphi, \psi > 0$ in $\Omega$ by \cite{vazquez1984}. Therefore, $\varphi, \psi \in S$ and for all $\xi, \eta \in \overline{S}$ we have
\begin{equation}
\label{sigm>inf1}
\int_\Omega |\nabla \varphi|^{p-2} \nabla \varphi \nabla\xi \, dx = 
\int_\Omega \xi \, dx \quad \mbox{and} \quad 
\int_\Omega |\nabla \psi|^{q-2} \nabla \psi \nabla\eta \, dx = 
\int_\Omega \eta \, dx.
\end{equation}
At the same time, there exists a constant $C_1 \in \mathbb{R}$ (possibly negative) such that
\begin{align}
\notag
\int_\Omega \xi \, dx &- c_1 \int_\Omega f\, |\varphi|^{\alpha-2} |\psi|^{\beta} \varphi \xi \, dx - C_1 \int_\Omega |\varphi|^{p-2} \varphi \xi \, dx\\
\label{sigm>inf2}
 &= \int_\Omega \left( 1 - c_1 f\, |\varphi|^{\alpha-2} |\psi|^{\beta} \varphi - C_1 |\varphi|^{p-2} \varphi \right) \xi \, dx \geq 0
\end{align}
for all $\xi \in \overline{S}$, since $\varphi, \psi, f$ are bounded and $|\varphi(x)|^{\alpha-2} |\psi(x)|^\beta \varphi(x) \to 0$ as $dist(x, \partial \Omega) \to 0$, due to $\alpha + \beta > 1$.
By a similar argument, there exists $C_2 \in \mathbb{R}$ such that for all $\eta \in \overline{S}$ we get
\begin{equation}
\label{sigm>inf3}
\int_\Omega \eta \, dx - c_2 \int_\Omega f\, |\varphi|^{\alpha-2} |\psi|^{\beta} \varphi \xi \, dx - C_2 \int_\Omega |\psi|^{q-2} \psi \eta \, dx \geq 0.
\end{equation}
Using now \eqref{sigm>inf1}, \eqref{sigm>inf2}, \eqref{sigm>inf3}, we conclude that
\begin{align*}
\lambda^*_e(r) \geq \inf_{\xi, \eta \in \overline{S}} 
\frac{ C_1 \int_\Omega |\varphi|^{p-2} \varphi \xi \, dx + C_2 \int_\Omega |\psi|^{q-2} \psi \eta \, dx}{ \int_\Omega |\varphi|^{p-2} \varphi \xi \, dx + r \int_\Omega |\psi|^{q-2} \psi \eta \, dx} \geq \min\{C_1, C_2 /r \} > -\infty
\end{align*}
for all $r>0$, where the penultimate inequality follows from Proposition \ref{lem:f}.

(2)
Assume $\Omega^+ \cup \Omega^0$ has nonempty interior a.e. Then one can find
an open set $B$ such that $B \subset \Omega^+ \cup \Omega^0$ after possible redefinition on a set of measure 0.
Let us fix any nontrivial nonnegative function $\varphi \in C^\infty_0(B)$ and consider $\varphi^{p}/u^{p-1}$ for each $u \in S$. Evidently, $\varphi^{p}/u^{p-1} \in C^1_0(\Omega)$, due to the necessary regularity of $\varphi, u$ and the facts that $u > 0$ in $\overline{B}$ and $\varphi \equiv 0$ in $\Omega \setminus B$.
Applying now Picone's identity (cf. \cite[Theorem 1.1]{Alleg}) for any $u \in S$  we get
\begin{equation*}
\int_\Omega |\nabla u|^{p-2} \nabla u \nabla \left( \frac{\varphi^{p}}{u^{p-1}} \right) \,dx 
\leq 
\int_\Omega |\nabla \varphi|^{p} \,dx \leq C_1 \int_\Omega |\varphi|^{p} \,dx,
\end{equation*}
where a constant $C_1 = C_1(\varphi)$ doesn't depend on $u$.
Hence, using the fact that $B \subset \Omega^+ \cup \Omega^0$ a.e., we get the following chain of inequalities:
\begin{align*}
\notag
\inf_{\xi, \eta \in \overline{S}} &\mathcal{L}_r(u, v; \xi, \eta) \\
\leq &\mathcal{L}_r \left(u, v; \frac{\varphi^{p}}{u^{p-1}}, 0 \right)
= \frac{ \int_\Omega |\nabla u|^{p-2} \nabla u \nabla \left( \frac{\varphi^{p}}{u^{p-1}} \right) \,dx  - c_1 \int_\Omega f |u|^{\alpha-p} |v|^\beta \varphi^{p} \,dx}
{\int_\Omega \varphi^{p} \,dx} \leq C_1 < +\infty
\end{align*}
for all $u, v \in S$ and $r \in [0, +\infty]$.
Consequently, we conclude that $\lambda^*_e(r) < +\infty$ on $[0, +\infty]$.

Likewise, we can find a constant $C_2 = C_2(\varphi)$, independent of $v$, such that for all $u, v \in S$ and $r \in [0, +\infty]$ it holds
\begin{align*}
\notag
r \, \inf_{\xi, \eta \in \overline{S}} &\mathcal{L}_r(u, v; \xi, \eta) \\
\leq r \, &\mathcal{L}_r \left(u, v; 0, \frac{\varphi^{q}}{v^{q-1}} \right)
= \frac{ \int_\Omega |\nabla v|^{q-2} \nabla v \nabla \left( \frac{\varphi^{q}}{v^{q-1}} \right) \,dx  - c_2 \int_\Omega f |u|^{\alpha} |v|^{\beta-q} \varphi^{q} \,dx}
{\int_\Omega \varphi^{q} \,dx} \leq C_2 < +\infty
\end{align*}
Hence, $\mu^*_e(r) < +\infty$ on $[0, +\infty]$, and statement $(2)$ of Theorem \ref{Th0} is proven.

(3)
Assume that $p=q$ and $\nu(\Omega^0 \cup \Omega^+) = 0$. Then  statements $(b)$ and $(e)$ of \cite[Theorem 1.1, p.~947]{ilrunSup} imply the existence of positive weak solution $u_\lambda \in C^{1, \delta}(\overline{\Omega})$ of the problem
\begin{equation}
\label{scalar}
\left\{
\begin{aligned}
  -\Delta_p u_\lambda &= \lambda |u_\lambda|^{p-2} u_\lambda + f(x) |u_\lambda|^{\alpha + \beta -2} u_\lambda , &&x \in \Omega, \\[0.4em]
 \left. u_\lambda \right|_{\partial \Omega} &=  0,
\end{aligned}
\right.
\end{equation}
for any $\lambda > \lambda_1$. At the same time, it is easy to see that $(u_\lambda, u_\lambda)$ becomes a positive $C^1$-solution of system \eqref{D1} with $p=q$, $\lambda = \mu$ and $c_1, c_2 = 1$, i.e.
\begin{equation*}
\left\{
\begin{aligned}
  -\Delta_p u_\lambda &= \lambda |u_\lambda|^{p-2} u_\lambda + f(x) |u_\lambda|^{\alpha-2} |u_\lambda|^{\beta}  u_\lambda , &&x \in \Omega, \\[0.4em]
 \, -\Delta_p u_\lambda &= \lambda |u_\lambda|^{p-2} u_\lambda + f(x) |u_\lambda|^{\alpha} |u_\lambda|^{\beta-2} u_\lambda, &&x \in \Omega, \\[0.4em]
 \left. u_\lambda \right|_{\partial \Omega} &= 0.
\end{aligned}
\right.
\end{equation*}
Proposition \ref{semivar} implies the existence of $t_0, s_0 > 0$, independent of $\lambda$, such that $(t_0 u_\lambda, s_0 u_\lambda)$ is a positive $C^1$-solution of \ref{D1}. Using $(t_0 u_\lambda, s_0 u_\lambda)$ as test point for $\mathcal{L}_r(\cdot, \cdot; \xi, \eta)$ we get
\begin{align*}
\lambda^*_e(r) \geq \inf_{\xi, \eta \in \overline{S}} \mathcal{L}_r(t_0 u_\lambda, s_0 u_\lambda; \xi, \eta) &=
\inf_{\xi, \eta \in \overline{S}} \frac{\lambda  \int_\Omega |t_0 u_\lambda|^{p-2} t_0 u_\lambda \xi \,dx + \lambda \int_\Omega |s_0 u_\lambda|^{p-2} s_0 u_\lambda \eta \,dx}{\int_\Omega |t_0 u_\lambda|^{p-2} t_0 u_\lambda \xi \,dx + r\, \int_\Omega |s_0 u_\lambda|^{p-2} s_0 u_\lambda \eta \,dx} \\
&\geq \lambda \,\min \left\{1, 1/r \right\} \to +\infty,
\end{align*}
as $\lambda \to +\infty$ for any $r>0$. Hence, $\lambda^*_e(r) = +\infty$ for all $r>0$, which completes the proof.

(4)
Let $r_0>0$ be such that $\lambda^*_f(r_0)>0$. Then the following set is nonempty 
\begin{equation}
\label{eq:S+}
\mathcal{S}^+_{r_0}:=\{(u,v)\in S\times S: \inf_{\xi, \eta \in \overline{S}} \mathcal{L}_{r_0}(u, v; \xi, \eta) \geq 0 \}.
\end{equation}
Observe that if $r \leq r_0$, then for all $(u,v) \in \mathcal{S}^+_{r_0}$ it holds 
\begin{equation*}
\inf_{\xi, \eta \in \overline{S}} \mathcal{L}_{r}(u, v; \xi, \eta) \geq
\inf_{\xi, \eta \in \overline{S}} \mathcal{L}_{r_0}(u, v; \xi, \eta).
\end{equation*}
Consequently,
\begin{equation}\label{sigma1}
\lambda^*_e(r) = \sup_{u, v \in S} \inf_{\xi, \eta \in \overline{S}} \mathcal{L}_{r}(u, v; \xi, \eta) \geq
\sup_{u, v \in S} \inf_{\xi, \eta \in \overline{S}} \mathcal{L}_{r_0}(u, v; \xi, \eta) = \lambda^*_e(r_0).	
\end{equation}
Thus, we have proved that if $\lambda^*_e(r_0 ) > 0$ for some $r_0>0$, then $\lambda^*_e(r) > 0$ for any $r \in (0,r_0]$.

If  now $r \geq r_0$, then  for any $(u,v) \in \mathcal{S}^+_{r_0}$ and $\xi, \eta \in \overline{S}$ it holds $\mathcal{L}_{r}(u, v; \xi, \eta) \,r  \geq 
\mathcal{L}_{r_0}(u, v; \xi, \eta) \, r_0$. Indeed,
\begin{align*}
&\mathcal{L}_{r}(u, v; \xi, \eta) \, r \\
&\frac{\int_\Omega |\nabla u|^{p-2} \nabla u \nabla \xi \, dx +  \int_\Omega |\nabla v|^{q-2} \nabla v \nabla \eta \, dx
- 
c_1 \int_\Omega f |u|^{\alpha-2} |v|^\beta u \xi \, dx -
c_2 \int_\Omega f |u|^{\alpha} |v|^{\beta-2} v \eta \, dx }
{\frac{1}{r}\int_\Omega |u|^{p-2} u \xi \, dx + \int_\Omega |v|^{q-2} v \eta \, dx}\\
&\geq
\frac{\int_\Omega |\nabla u|^{p-2} \nabla u \nabla \xi \, dx +  \int_\Omega |\nabla v|^{q-2} \nabla v \nabla \eta \, dx
- 
c_1 \int_\Omega f |u|^{\alpha-2} |v|^\beta u \xi \, dx -
c_2 \int_\Omega f |u|^{\alpha} |v|^{\beta-2} v \eta \, dx }
{\frac{1}{r_0}\int_\Omega |u|^{p-2} u \xi \, dx + \int_\Omega |v|^{q-2} v \eta \, dx} \\
&= \mathcal{L}_{r_0}(u, v; \xi, \eta) \, r_0.
\end{align*}
Thus, for all $r \geq r_0$ we find that
\begin{equation}\label{sigma2}
\mu^*_e(r) = \lambda^*_e(r) \, r = \sup_{u, v \in S} \inf_{\xi, \eta \in \overline{S}} \mathcal{L}_{r}(u, v; \xi, \eta) \, r \geq
\sup_{u, v \in S} \inf_{\xi, \eta \in \overline{S}} \mathcal{L}_{r_0}(u, v; \xi, \eta) \, r_0 = \lambda^*_e(r_0) \, r_0 = \mu^*_e(r_0).
\end{equation}
Consequently, if $\lambda^*_e(r_0) > 0$ for some $r_0>0$, then $\lambda^*_e(r) \geq \lambda^*_e(r_0)r_0/r>0$ for any $r\in [r_0, +\infty)$. Therefore, we have $\lambda^*_e(r)>0$ for all $r \in (0,\infty)$, and hence $\Gamma_e(r)>(0, 0)$.

(5) 
Let $r_1 > 0$ be such that  $\lambda^*_e(r_1) < +\infty$.  If $\mathcal{S}^+_r$ given by \eqref{eq:S+} is empty for any $r>0$, then evidently $\lambda^*_e(r) \leq 0$ for all $r>0$ and the assertion of the theorem is true.
Assume now that there exists $r_0 \neq r_1$, such that $\mathcal{S}^+_{r_0} \neq \emptyset$.
Then inequalities \eqref{sigma1} and \eqref{sigma2} imply that $\lambda^*_e(r_0) \leq \lambda^*_e(r_1) < +\infty$ or $\lambda^*_e(r_0) \leq \lambda^*_e(r_1) \, r_1/r_0 < +\infty$, respectively. 
Therefore, $\Gamma_e(r) < +\infty$ for all $r>0$, and the proof is complete.

(7)
Let $(0, 0) < \Gamma_e(r) < +\infty$, $r \in (0,\infty)$, then $\lambda^*_e(r)$ is nonincreasing and $\mu^*_e(r)$ is nondecreasing on $(0, +\infty)$, due to \eqref{sigma1} and \eqref{sigma2}, respectively.

(6)
The desired continuity of $\Gamma_e$ can be proved using the monotonicity \eqref{sigma1} and \eqref{sigma2} in much the same way as statement (2) of Theorem \ref{Th3}.
\qed

Observe that $(\overline{u}, \overline{v}) \in S\times S$ is a positive $C^1$-supersolution  of \eqref{D1} with $(\lambda, \mu)$
if and only if
\begin{equation*}
 \Phi_{(\lambda,\mu)}(\overline{u}, \overline{v}; \xi,\eta) \geq 0,  \quad \forall~ (\xi,\eta) \in \overline{S}\times \overline{S},
\end{equation*}
or, equivalently,
\begin{equation}
\label{C1solution}
\mathcal{L}_r(\overline{u}, \overline{v}; \xi, \eta) \geq \lambda, \quad \forall~ (\xi,\eta) \in \overline{S}\times \overline{S} \setminus \{(0, 0)\}, 
\end{equation}
where $r = \mu/\lambda$.


\medskip
\noindent
\textit{Proof of Theorem \ref{Th1}.}
(1) Fix any $(\lambda, \mu) \in \Sigma^*_e$ and let $r_0 = \mu/\lambda$.
We will obtain the proof, if we show that \eqref{C1solution} holds for some $(\overline{u}, \overline{v}) \in S \times S$.
To this end, let us prove that $\lambda < \lambda^*(r_0)$ and  $\mu < \mu^*(r_0)$. Evidently, it is sufficient to check only the first inequality.
Suppose, contrary to our claim, that $\lambda \geq \lambda^*(r_0)$. Then $(\lambda, \mu) \in \Gamma_e$ if the equality holds, or $(\lambda, \mu) \in \mathcal{R}$ if $\lambda > \lambda^*(r_0)$. 
However, from above (see Subsection 2.2) we know that $\Sigma^*_e \cap \mathcal{R} = \emptyset$ and $\Gamma_e$ separates these sets. Hence, we get a contradiction to our assumption $(\lambda, \mu) \in \Sigma^*_e$.
Thus, the definition of \eqref{sf} implies the existence of $(\overline{u}, \overline{v}) \in S \times S$ such that 
$$
\lambda < \inf_{\xi, \eta \in \overline{S}} \mathcal{L}_{r_0}(\overline{u}, \overline{v}; \xi, \eta) \leq \lambda^*_e(r_0).
$$
Hence $\lambda < \mathcal{L}_{r_0}(\overline{u}, \overline{v}; \xi, \eta)$ for all $(\xi,\eta) \in \overline{S}\times \overline{S}$, and therefore $(\overline{u}, \overline{v})$ is a positive $C^1$-supersolution of \eqref{D1} with $(\lambda, \mu)$.

(2)
Let $\mathcal{R} \neq \emptyset$.
Suppose, contrary to our claim, that there exists a positive $C^1$-supersolution $(u, v)$ of \eqref{D1} for some $(\lambda, \mu) \in \mathcal{R}$. 
Arguing as in statement (1), it can be proved that $\lambda^*_e(r_0) < \lambda$ and $\mu^*_e(r_0) < \mu$ for $r_0 = \mu/\lambda$.
Hence, 
$$
\lambda^*_e(r_0) < \lambda \leq \inf_{\xi, \eta \in \overline{S}}\mathcal{L}_{r_0}(u, v; \xi, \eta) \leq \lambda^*_e(r_0),
$$
which is impossible.
\qed

\begin{remark}
Statement (1) of Theorem \ref{Th0} remains valid for\begin{equation*}
(\lambda, \mu) \in \bigcup_{r>0} \left\{ (\lambda, \mu) \in \mathbb{R}^2: ~\lambda < \lambda^*_e(r), ~\mu < \mu^*_e(r) \right\}.
\end{equation*}
\end{remark}

Let us now prove Lemma \ref{invar}. To reflect the dependence of the problem on the constants $c_1,c_2 \in \mathbb{R}^+$, we will temporarily use the notations \eqref{D1}$(c_1,c_2)$, $\Gamma_e(c_1, c_2)$, $\mathcal{R}(c_1,c_2)$, $\Sigma(c_1,c_2)$, etc. 

\noindent
\textit{Proof of Lemma \ref{invar}.}
Assume, contrary to our claim, there exists $r_0 > 0$ such that $\Gamma_e(c_1, c_2)(r_0) \neq \Gamma_e(d_1, d_2)(r_0)$ for some $c_1, c_2, d_1, d_2 > 0$. 
Since we use the parametrization of $\Gamma_e$ by rays $(\lambda, \lambda \, r_0)$, 
we may assume, without loss of generality, that  $\lambda_e^*(c_1, c_2)(r_0) < \lambda_e^*(d_1, d_2)(r_0)\leq +\infty$ and, consequently, $\mu_e^*(c_1, c_2)(r_0) < \mu_e^*(d_1, d_2)(r_0)\leq +\infty$. Statement (2) of Theorem \ref{Th1} implies that \eqref{D1}$(d_1,d_2)$ possesses a positive $C^1$-supersolution for any $(\lambda,\mu) $ such that $\lambda < \lambda_e^*(d_1, d_2)(r_0)$, $\mu < \mu_e^*(d_1, d_2)(r_0)$. Hence, Proposition \ref{semivar} yields that \eqref{D1}$(c_1,c_2)$ has also a positive $C^1$-supersolution for the same $(\lambda,\mu)$. However, by statement (1) of Theorem \ref{Th1}, problem
\eqref{D1}$(c_1,c_2)$ has no positive $C^1$-supersolutions for $\lambda>\lambda_e^*(c_1, c_2)(r_0)$ and $\mu>\mu_e^*(c_1, c_2)(r_0)$. This contradicts our assumption.
\qed

To prove the following fact let us note that $\overline{S}$ is the positive cone of the Banach space $\{w \in C^1(\overline{\Omega}):~w = 0 \mbox{ on } \partial \Omega \}$. 
Define an interior of $\overline{S}$ w.r.t. to $C^1$-topology as
\begin{equation}
\label{def:int}
\mathrm{int}\, \overline{S} := \left\{ w \in C^1(\overline{\Omega}):
~w>0 \text{ in } \Omega, 
~w = 0 \mbox{ on } \partial \Omega,
~\frac{\partial w}{\partial n} < 0 \text{ on } \partial\Omega \right\}, 
\end{equation}
where $n$ is the unit outward normal vector to $\partial \Omega$. 
Notice also that $\mathrm{int}\, \overline{S} \subset S$.

\begin{lemma}
\label{lem:solution}
Assume that for some $r > 0$ there exist a maximizer $(u^*, v^*) \in \overline{S} \times \overline{S}$ and a corresponding minimizer $(\xi^*, \eta^*) \in \mathrm{int}\, \overline{S} \times \mathrm{int}\, \overline{S}$ of \eqref{extcrit}, i.e.
$$
\lambda^*_e(r) = \inf_{\xi, \eta \in \overline{S}} \mathcal{L}_r(u^*, v^*, \xi, \eta) = \mathcal{L}_r(u^*, v^*, \xi^*, \eta^*).
$$
Then $(u^*, v^*)$ is a weak solution of \eqref{D1}.
\end{lemma}
\begin{proof}
Since $\inf_{\xi, \eta \in \overline{S}} \mathcal{L}_r(u^*, v^*, \xi, \eta)$ is attained at $(\xi^*, \eta^*) \in \mathrm{int}\, \overline{S} \times \mathrm{int}\, \overline{S}$, we have
$$
D_\xi \mathcal{L}_r(u^*, v^*, \xi^*, \eta^*) = 0, \quad 
D_\eta \mathcal{L}_r(u^*, v^*, \xi^*, \eta^*) = 0,
$$
which is equivalent to
\begin{equation}
\label{solext}
  \begin{aligned}
\int_\Omega |\nabla u^*|^{p-2} \nabla u^* \nabla \xi \,dx - \lambda^*_e(r) \int_\Omega |u^*|^{p-2} u^* \xi \,dx - c_1 \int_\Omega f |u^*|^{\alpha-2} |v^*|^{\beta} u^* \xi \,dx &=  0,\\
\int_\Omega |\nabla v^*|^{q-2} \nabla v^* \nabla \eta \,dx- 
\mu^*_e(r) \int_\Omega |v^*|^{q-2} v^* \eta \,dx - c_2 \int_\Omega f |u^*|^{\alpha} |v^*|^{\beta-2} v^* \eta \,dx &=  0,
 \end{aligned}
\end{equation}
for all $\xi, \eta \in \overline{S}$.  Thus, we obtain the desired conclusion.
\end{proof}

From \eqref{solext} it follows that 
$\lambda^*_e(r) = \mathcal{L}_r(u^*, v^*, \xi, \eta)$
for all $\xi, \eta \in \overline{S}$. 
Thus, the existence of a minimizer $(\xi^*, \eta^*) \in \mathrm{int}\, \overline{S} \times \mathrm{int}\, \overline{S}$ implies that all $(\xi, \eta) \in \overline{S} \times \overline{S}$ are minimizers of \eqref{extcrit}.
On the other hand, in general, the functional
$\inf_{\xi, \eta \in \overline{S}} \mathcal{L}_r(u^*, v^*, \xi, \eta)$ is not differentiable with respect to $u^*$ and $v^*$. Therefore, we cannot conclude that 
$$
D_u \mathcal{L}_r(u^*, v^*, \xi, \eta) = 0, \quad 
D_v \mathcal{L}_r(u^*, v^*, \xi, \eta) = 0,
$$
for all $(\xi, \eta) \in \overline{S} \times \overline{S}$. 
However, it easy to see that at least in the case $p,q=2$, if these equalities hold for some $\xi^*, \eta^*$, then $\xi^*, \eta^*$ belong to the kernel of the corresponding linearized operator, i.e.
\begin{equation*}
\label{xxx}
\left\{
\begin{aligned}
  -\Delta \xi^*  &- \lambda^*_e(r) \xi^* - c_1 (\alpha-1) f(x) |u^*|^{\alpha-2} |v^*|^{\beta} \xi^* - c_1 \beta f(x) |u^*|^{\alpha-2} |v^*|^{\beta-2} u^* v^* \eta^* = 0, &&x \in \Omega, \\[0.4em]
  -\Delta \eta^* &- \mu^*_e(r) \eta^* - c_2 (\beta - 1) f(x) |u^*|^{\alpha} |v^*|^{\beta-2} \eta^* - c_2 \alpha f(x) |u^*|^{\alpha-2} |v^*|^{\beta-2} u^* v^* \xi^* = 0, &&x \in \Omega, \\[0.4em]
	\left. \xi^* \right|_{\partial \Omega} &=  \left. \eta^* \right|_{\partial \Omega} = 0.
\end{aligned} \right.
\end{equation*}

\medskip
Finally, we prove nonexistence results.

\par\noindent  
{\it Proof of Lemma \ref{Th27}}.

\par\noindent
(1) 
Assume $\nu(\Omega^+) = 0$. 
Suppose, contrary to our claim, that there exists a nontrivial weak solution $(u, v) \in W_0^{1,p} \times W_0^{1,q}$ of \eqref{D1} in $\mathbb{R}^2\setminus\{ \lambda \geq \lambda_1\}\times\{\mu \geq \mu_1 \}$. 
Without loss of generality, we can assume $\lambda < \lambda_1$.
Using this facts, we test the first equation in \eqref{D1} by $u$ and get 
\begin{align*}
0 < \int_\Omega |\nabla u|^{p} \,dx - \lambda \int_\Omega |u|^{p} \,dx = c_1 \int_\Omega f |u|^{\alpha} |v|^{\beta} \,dx \leq 0.
\end{align*}
Thus we obtain a contradiction.

\par\noindent
The proof of statements (2) and (3) directly follows from the proof of statement (2) of Theorem \ref{Th0}.

\qed

\appendix
\section{Additional properties}

We use the temporary notation \eqref{D1}$(c_1,c_2)$ to reflect the dependence of \eqref{D1} on parameters $c_1, c_2$.

\begin{proposition}
\label{semivar}
Let $p, q>1$, $\lambda, \mu, \alpha, \beta \in \mathbb{R}$ and $\frac{\alpha}{p} + \frac{\beta}{q} - 1 \neq 0$. If $(u, v) \in W_0^{1,p} \times W_0^{1,q}$ is a weak solution (sub-, supersolution) of \eqref{D1}$(c_1,c_2)$ with some $c_1, c_2 > 0$, then for any $d_1, d_2 > 0$ there exist $t, s \in (0, +\infty)$ such that 
$(t u, s v) \in W_0^{1,p} \times W_0^{1,q}$ is a weak solution (sub-, supersolution) of \eqref{D1}$(d_1,d_2)$.  
\end{proposition}
\begin{proof}
Let $(u, v) \in W_0^{1,p} \times W_0^{1,q}$ be a weak solution of \eqref{D1}$(c_1, c_2)$ with some $c_1, c_2 > 0$. 
Multiplying the first equation of \eqref{D1}$(c_1, c_2)$ by $t^{p-1}$ and the second equation by $s^{q-1}$, where $t, s > 0$, we get (in the weak sense)
\begin{equation*}
\left\{
\begin{aligned}
  -\Delta_p (t u) &= \lambda |t u|^{p-2} t u + c_1 \, t^{p - \alpha} s^{- \beta} f(x) |t u|^{\alpha-2} |s v|^{\beta}  t u, \\[0.4em]
 \, -\Delta_q (s v) &= \mu |s v|^{q-2} s v + c_2 \, t^{- \alpha} s^{q - \beta} f(x) |t u|^{\alpha} |s v|^{\beta-2} s v.
\end{aligned}
\right.
\end{equation*}
Let us fix any $d_1, d_2 > 0$ and find $t, s$ such that
$$
c_1 \, t^{p - \alpha} s^{- \beta} = d_1,
\quad 
c_2 \, t^{- \alpha} s^{q - \beta} = d_2.
$$
Then 
$$
t = \left( \frac{c_1}{d_1} \right)^{\frac{q-\beta}{p q d}} \left( \frac{c_2}{d_2} \right)^{\frac{\beta}{p q d}},
\quad  
s = \left( \frac{c_1}{d_1} \right)^{\frac{\alpha}{p q d}} \left( \frac{c_2}{d_2} \right)^{\frac{p - \alpha}{p q d}},
$$
where $d :=  \frac{\alpha}{p} + \frac{\beta}{q} - 1 \neq 0$, by the assumption. 
Hence $t, s \in (0, +\infty)$, since $c_1, c_2$, $d_1, d_2 > 0$, and $(t u, s v)$ satisfies \eqref{D1}$(d_1, d_2)$ in the weak sense.
The converse assertion can be proved by the similar way.
\end{proof}

Consider now the function
$$
g(x, y) := \frac{a x + b y}{c x + d y}, \quad x, y \geq 0, \quad (x,y) \neq (0,0),
$$
with $a,b \in \mathbb{R}$, $c, d > 0$.
\begin{proposition}
\label{lem:f}
$\min \{ \frac{a}{c}, \frac{b}{d}\} \leq g(x, y) \leq \max \{ \frac{a}{c}, \frac{b}{d}\}$ for all $x, y \geq 0$, $(x,y) \neq (0,0)$.
\end{proposition}
\begin{proof}
Without loss of generality, we may suppose that $x \neq 0$. Consider the function
$$
h(s) := \frac{a + b s}{c + d s}, \quad s \geq 0.
$$
Evidently, $h(y/x) = f(x,y)$.
It is not hard to show that $h(s)$ is monotone for $s \geq 0$.
Therefore, the extremal values of $h$ on $[0, +\infty)$ will be achieved either for $s = 0$ or $s = +\infty$.
Finding the corresponding limits of $h(s)$ we obtain the desired result.
\end{proof}

\section{Regularity}

The next lemma provides the boundedness of weak solutions to \eqref{D1}.
This result is sufficient to obtain $C^{1}(\overline{\Omega})$-regularity of solutions and, in addition, the maximum principle for nonnegative ones.
The proof is based on the well-known \textit{bootstrap arguments} (cf. \cite[Lemma 3.2, p.~114]{drabekbook}).
\begin{lemma}
\label{lem:l}
Assume $\alpha, \beta > 0$, $\frac{\alpha}{p^*} + \frac{\beta}{q^*} < 1$ and $\lambda, \mu \in \mathbb{R}$. Let  $(u, v) \in W_0^{1,p} \times W_0^{1,q}$
be a weak solution  of \eqref{D1}. Then $(u, v) \in L^\infty(\Omega)\times L^\infty(\Omega)$.
\end{lemma}
\begin{proof}
Let first $(u, v) \in W_0^{1,p} \times W_0^{1,q}$ be a \textit{nonnegative} weak solution of \eqref{D1}.
Define $u_M := \min\{ u, M \}$ in $\Omega$, where $M > 0$. Then $u_M^{k p + 1} \in W_0^{1,p}$ for any $M, k > 0$ and $p > 1$. Indeed,
\begin{align*}
&\int_\Omega \left|\nabla \left( u_M^{k p + 1} \right) \right|^p \,dx = (k p +1)^p \int_\Omega u_M^{k p^2}  \left|\nabla u_M \right|^p  \,dx \\
&=(k p +1)^p \int_{\{u \leq M \}} u^{k p^2}  \left|\nabla u \right|^p  \,dx \leq 
(k p +1)^p M^{k p^2} \int_\Omega |\nabla u|^p \,dx < +\infty.
\end{align*}
Testing the fist equation of \eqref{D1} by $u_M^{k p + 1} \in W_0^{1,p}$ we obtain
\begin{equation}
\label{123}
\int_\Omega |\nabla u|^{p-2} \nabla u \nabla \left( u_M^{k p + 1} \right) \,dx = 
\lambda \int_\Omega |u|^{p-2} u u_M^{k p + 1} \,dx + c_1 \int_\Omega f |u|^{\alpha-2} |v|^{\beta} u u_M^{k p + 1} \,dx.
\end{equation}
Using the Sobolev embedding theorem, for the first integral we have the following estimation:
\begin{align}
\notag
&\int_\Omega |\nabla u|^{p-2} \nabla u \nabla \left( u_M^{k p + 1} \right) \,dx = 
(k p + 1) \int_\Omega u_M^{k p} \, |\nabla u_M|^{p} \,dx \\
&= \frac{(k p + 1)}{(k + 1)^p} \int_\Omega \left|\nabla \left( u_M^{k + 1} \right) \right|^{p} \,dx \geq
\label{firstint}
C_1 \frac{(k p + 1)}{(k + 1)^p} \left( \int_\Omega \left| u_M^{k + 1} \right|^{p^*} \,dx \right)^{\frac{p}{p^*}}.
\end{align}
Note that $C_1 > 0$ is independent of $M$ and $k$.
If $\lambda > 0$, taking \textit{any} $t \in (p, p^*)$ and applying H\"older's inequality for the second integral in \eqref{123}, we get
\begin{equation}
\label{secondint}
\int_\Omega |u|^{p-2} u u_M^{k p + 1} \,dx \leq  \int_\Omega |u|^{p(k+1)} \,dx \leq
C_2 \left( \int_\Omega |u|^{t (k+1)}  \,dx \right)^{\frac{p}{t}}.
\end{equation}
If $\lambda \leq 0$, then we have $\lambda \int_\Omega |u|^{p-2} u u_M^{k p + 1} \,dx \leq 0$.

Since $f \in L^\infty(\Omega)$, the third integral in \eqref{123} can be estimated initially as follows:
\begin{equation}
\label{last3}
\int_\Omega f |u|^{\alpha-2} |v|^{\beta} u u_M^{k p + 1} \,dx \leq C_3 \int_\Omega |u|^{k p + \alpha} |v|^{\beta} \,dx.
\end{equation}
Suppose fist $\alpha = p$. 
Due to the subcriticial assumption $\frac{\alpha}{p^*} + \frac{\beta}{q^*} < 1$, there exists $t \in (p, p^*)$ such that $\frac{\alpha}{t} + \frac{\beta}{q^*} \leq 1$ and consequently $\frac{p}{t} + \frac{\beta}{q^*} \leq 1$. Therefore, applying the H\"older inequality to the right-hand side of \eqref{last3}, we derive
\begin{align}
\label{kpa0}
\int_\Omega |u|^{p (k+1)} |v|^{\beta} \,dx  \leq C_4
\left( \int_\Omega |u|^{t (k+1)} \,dx \right)^{\frac{p}{t}} 
\left( \int_\Omega |v|^{q^*} \,dx \right)^{\frac{\beta}{q^*}} \leq C_5 \left( \int_\Omega |u|^{t (k+1)} \,dx\right)^{\frac{p}{t}},
\end{align}
where $C_5$ depends on $v$, but does not depend on $k$ and $M$.

Suppose now $\alpha < p$. 
Using \eqref{kpa0}, the right-hand side of \eqref{last3} can be estimated as follows:
\begin{align}
\notag
\int_\Omega |u|^{k p + \alpha} |v|^{\beta} \,dx &= 
\int_{\{ u < 1 \}} |u|^{k p + \alpha} |v|^{\beta} \,dx + 
\int_{ \{ u \geq 1 \}} |u|^{k p + \alpha} |v|^{\beta} \,dx \\
&
\notag
\leq \int_{\{ u < 1 \}} |v|^{\beta} \,dx + 
\int_{ \{ u \geq 1 \}} |u|^{k p + p} |v|^{\beta} \,dx \\ 
\label{kpa1} 
&\leq
C_6 + C_5 \left( \int_\Omega |u|^{t (k+1)} \,dx \right)^{\frac{p}{t}} \leq
C_7 \left( \int_\Omega |u|^{t (k+1)} \,dx \right)^{\frac{p}{t}},
\end{align}
where the last inequality is obtained by estimation
$$
\int_\Omega |u|^{t (k+1)} \,dx \geq 
\int_{ \{ u \geq 1 \}} |u|^{t (k+1)} \,dx \geq
\nu(\{ u \geq 1 \}) = C_8 > 0.
$$
Notice that although $C_6$, $C_7$ depend on $u, v$, they are independent of $k$ and $M$.

Suppose finally $\alpha > p$. 
Note that
$$
\frac{\alpha}{p^*} + \frac{\beta}{q^*} < 1 ~ \Longrightarrow ~ \exists \,t \in (p, p^*): ~ 
\frac{\alpha}{p^*} + \frac{\beta}{q^*} -p \left( \frac{1}{p^*} - \frac{1}{t} \right) \leq 1
 ~ \Longrightarrow ~ 
\frac{\alpha - p}{p^*} + \frac{\beta}{q^*} + \frac{p}{t} \leq 1.
$$
Therefore, using the H\"older inequality, we get
\begin{align}
\notag
&\int_\Omega |u|^{k p + \alpha} |v|^{\beta} \,dx  = \int_\Omega |u|^{p (k+1)} |u|^{\alpha - p} |v|^{\beta} \,dx  \\
\label{kpa}
&\leq 
C_9 \left( \int_\Omega |u|^{t (k+1)} \,dx \right)^{\frac{p}{t}} 
\left( \int_\Omega |u|^{p^*} \,dx \right)^{\frac{\alpha - p}{p^*}}
\left( \int_\Omega |v|^{q^*} \,dx \right)^{\frac{\beta}{q^*}} \leq
C_{10} \left( \int_\Omega |u|^{t (k+1)} \,dx \right)^{\frac{p}{t}}.
\end{align}
Combining  \eqref{last3} with \eqref{kpa0}, \eqref{kpa1} or \eqref{kpa}, we derive
\begin{equation}
\label{fuv}
\int_\Omega f |u|^{\alpha-2} |v|^{\beta} u_M^{k p + 1} \,dx \leq C_{11} \left( \int_\Omega |u|^{t (k+1)} \,dx \right)^{\frac{p}{t}},
\end{equation}
where constant $C_{11} \in (0, +\infty)$ depends on $u$ and $v$, but does not depend on $k$ and $M$.

Substituting estimations \eqref{firstint}, \eqref{secondint} and \eqref{fuv} into the  energy equation \eqref{123}, we get
\begin{equation*}
\left( \int_\Omega \left| u_M \right|^{p^* (k+1)} \,dx \right)^{\frac{1}{p^* (k+1)}} \leq 
C_{12}^{\frac{1}{k+1}} \left( \frac{(k + 1)}{(k p + 1)^{1/p}} \right)^{\frac{1}{k+1}} \left( \int_\Omega |u|^{t (k+1)} \,dx \right)^{\frac{1}{t (k+1)}}.
\end{equation*}
Taking $k_1 > 0$ such that $t (k_1 + 1) = p^*$  and passing to the limit as $M \to +\infty$, we obtain
\begin{equation}
\label{first}
\left( \int_\Omega \left| u \right|^{p^* (k_1 + 1)} \,dx \right)^{\frac{1}{p^* (k_1 + 1)}} \leq 
C_{12}^{\frac{1}{k_1 + 1}} \left( \frac{(k_1 + 1)}{(k_1 p + 1)^{1/p}} \right)^{\frac{1}{k_1 + 1}} \left( \int_\Omega |u|^{p^*} \,dx \right)^{\frac{1}{p^*}} < +\infty.
\end{equation}
Therefore, $u \in L^{p^*(k_1+1)}(\Omega)$. Organizing now the iterative process exactly as in   \cite[Lemma 3.2, p.~114]{drabekbook}, we conclude that $u \in L^\infty(\Omega)$.
The same reasoning is applied to show that $v \in L^\infty(\Omega)$. Hence, the nonnegative solutions of \eqref{D1} are bounded. To prove the boundedness of nodal solutions we apply the above arguments separately to positive and negative parts. It is possible, due to the fact that if $u \in W_0^{1,p}(\Omega)$, then $\max\{ \pm u, 0 \} \in W_0^{1,p}(\Omega)$, see \cite[Corollary A.5, p.~54]{Stamm}.
\end{proof}

\begin{corollary}
\label{rem:c} 
Assume $\alpha, \beta \geq 1$. Let  $(u, v) \in W_0^{1,p} \times W_0^{1,q}$
be a bounded weak solution  of \eqref{D1}. Then $u, v \in C^{1,\gamma}(\overline{\Omega})$, $\gamma \in (0,1)$. 
\end{corollary}
\begin{proof}
The assumption $\alpha, \beta \geq 1$ and the boundedness of $(u,v)$ imply that the right-hand side of \eqref{D1} is bounded. Therefore, the regularity result of \cite{Lieberman} implies $u, v \in C^{1,\gamma}(\overline{\Omega})$ for some $\gamma \in (0,1)$.
\end{proof}

Corollary \ref{rem:c} and the strong maximum principle \cite{vazquez1984} imply
\begin{corollary}
\label{rem:>0}
Assume that $\alpha \geq p$, $\beta \geq q$ and let $(u, v) \in W_0^{1,p} \times W_0^{1,q}$ be a nontrivial nonnegative bounded weak solution of \eqref{D1}. Then $u$ and $v$ are \textit{positive} in $\Omega$. Moreover, they satisfy a boundary point maximum principle on $\partial \Omega$.
\end{corollary}


\end{document}